\documentclass[11pt, oneside, reqno]{amsart}
\usepackage{amsmath}
\usepackage{amsthm}
\usepackage{amssymb} 
\usepackage[mathscr]{euscript}
\usepackage{comment}
\usepackage{amsaddr}

\usepackage{fancyhdr}
\usepackage{enumerate}
\usepackage{amscd}
\usepackage[all]{xy}
\usepackage{graphicx}
\usepackage{hyperref}
\usepackage{tikz}
\usepackage[all]{xy}
\usepackage{graphicx}
\usetikzlibrary{backgrounds}
\usepackage{accents}

\newcommand{\Si}{\Sigma}

\newcommand{\be}{\begin{equation}}
\newcommand{\ee}{\end{equation}}
\newcommand{\bes}{\begin{equation*}}
\newcommand{\ees}{\end{equation*}}

\renewcommand{\to}{\rightarrow}

\newcommand{\Ric}{\mathrm{Ric}}

\theoremstyle{plain}
\newtheorem{lemma}{Lemma}[section]
\newtheorem{proposition}[lemma]{Proposition}
\newtheorem{theorem}[lemma]{Theorem}
\newtheorem{Theorem}{Theorem}

\newtheorem{Corollary}[Theorem]{Corollary}

\newtheorem{conjecture}[Theorem]{Conjecture}

\numberwithin{equation}{section}

\theoremstyle{remark}
\theoremstyle{definition}
\newtheorem{remark}[lemma]{Remark}

\newtheorem{Definition}[Theorem]{Definition}

\DeclareMathOperator{\tr}{tr} 
\DeclareMathOperator{\Div}{div}

\DeclareMathOperator{\da}{d\sigma}

\DeclareMathOperator{\Ker}{Ker}

\DeclareMathOperator{\Int}{Int}

\DeclareMathOperator{\C}{\mathcal{C}}

\usepackage[tmargin=1in,bmargin=1in,lmargin=1in,rmargin=1in]{geometry}
\begin{document}

\title[New asymptotically flat static vacuum metrics]
%Remarks on ``Existence of static vacuum extensions'']{
{New asymptotically flat static vacuum metrics with near Euclidean boundary data}
%Remarks on ``Existence of static vacuum extensions with  prescribed Bartnik boundary data"}

\author{Zhongshan An and Lan-Hsuan Huang}

\address{Department of Mathematics, University of Connecticut, Storrs, CT 06269, USA}
\email{zhongshan.an@uconn.edu}
\email{lan-hsuan.huang@uconn.edu}

\date{\today}

\begin{abstract}
In our prior work toward  Bartnik's static vacuum extension conjecture for near Euclidean boundary data, we establish a sufficient condition, called static regular, and confirm large classes of boundary hypersurfaces are static regular. In this note, we further improve some of those prior results. Specifically, we show that any hypersurface in an open and dense subfamily of a certain general smooth one-sided family of hypersurfaces  (not necessarily a foliation) is static regular. The proof uses some of our new arguments motivated from studying the conjecture for boundary data near an arbitrary static vacuum metric.  
\end{abstract}

\maketitle

\section{Introduction}

Let $n\ge 3$ and $(M, g)$ be an $n$-dimensional Riemannian manifold. We say that $(M, g)$ is \emph{static vacuum} (or $g$ is a \emph{static vacuum metric} on $M$) if there is a scalar-valued function $u$ on $M$ satisfying
\begin{align}\label{eq:static}
\begin{split}
	-u\mathrm{Ric}_g + \nabla^2_g u&=0\\
	\Delta_g u&=0.
\end{split}
\end{align}	
Such $u$ is called a \emph{static potential}. The class of static vacuum metrics has played a fundamental role in general relativity because when $u>0$, the triple $(M, g, u)$ gives rise to a Ricci flat spacetime $(\mathbb R\times M, -u^2 dt^2 + g)$ that has a global Killing vector field $\partial_t$. 

%Let the number $q\in (\frac{n-2}{2}, n-2)$. We say that $(M, g)$ is \emph{asymptotically flat} (at the rate $q$) if $M\setminus  B$ is diffeomorphic to $\mathbb{R}^n \setminus B_1$ where $B, B_1$ are compact subsets and that  the metric $g$ has the asymptotics $g_{ij} =\delta_{ij} + O(|x|^{-q})$ with respect to the pull-back Cartesian coordinates $\{ x_1, \dots, x_n\}$ on $M\setminus K$.

A very important example of asymptotically flat, static vacuum metrics is the the family of  (Riemannian) Schwarzschild metrics $g_m$:
 \begin{align*}
 	g_m &= \left(1-\tfrac{2m}{r^{n-2}}\right)^{-1} dr^2 + r^2 g_{S^{n-1}} \quad \mbox{ defined on }  \mathbb{R}^n\setminus B_{(2m)^{\frac{1}{n-2}}},
\end{align*} 
with the static potential $u_m = \sqrt{1-\tfrac{2m}{r^{n-2}}}$, where $g_{S^{n-1}} $ is the standard metric on the unit sphere $S^{n-1}$. Note that the Schwarzschild metrics are rotationally symmetric. When $m=0$, the Schwarzschild metric becomes the Euclidean metric. When $m>0$, the Schwarzschild manifold has a minimal hypersurface boundary, precisely at $u=0$. In fact, the Schwarzschild metrics are the only asymptotically flat, static vacuum 3-manifolds with such property, by the celebrated Uniqueness Theorem of Static Black Holes. See \cite{Israel:1968, Robinson:1977, Bunting-Masood-ul-Alam:1987}. Another family of static vacuum, exact solutions was discovered by H.~Weyl. The Weyl solutions are axially symmetric and have general asymptotics at infinity, but a subclass of them can have asymptotically flat end. Those exact solutions can be characterized by certain conditions (e.g. having black hole boundary), and great efforts have been made toward the uniqueness and classification results of those static vacuum metrics. See, for example, M. Reiris and J.~Peraza~\cite{Reiris:2019}  and the references therein. 

In contrast, Robert Bartnik conjectured the following ``prescribing boundary value'' problem for asymptotically flat, static vacuum manifolds~\cite[Conjecture 7]{Bartnik:2002}.\footnote{Note that the original conjecture was  stated for $n=3$, $\Omega= B^3$, and $M=\mathbb R^3\setminus B^3$.} The conjecture was originated from his quasi-local mass program in 1989, for which we refer the reader to  the  survey article of M.~Anderson \cite{Anderson:2019} for details. The conjecture itself is also of independent interest as a natural geometric PDE boundary value problem. Furthermore, progress toward this conjecture would give rise to new examples of asymptotically flat, static vacuum metrics and advance our understanding toward the structure of static vacuum metrics. 

\begin{conjecture}[Static Extension Conjecture]\label{static}
Let $(\Omega, g_0)$ be a compact manifold with scalar curvature $R_{g_0}\ge 0$. Suppose the mean curvature $H_{g_0}$ is positive somewhere on the boundary $\Sigma$. Then there exists  a unique asymptotically flat, static vacuum manifold $(M, g)$  with boundary $\partial M \cong \Sigma$ satisfying
\begin{align*}
	\begin{array}{ll}	g_0^\intercal = g^\intercal \\
	H_{g_0} = H_g 
	\end{array} \quad \mbox{ on } \Sigma. 
\end{align*}
Here, $(\cdot)^\intercal$ denotes the restriction on the tangent bundle of $\Sigma$. 
\end{conjecture}

\noindent{\bf Convention:} The mean curvature $H_g$ of a hypersurface $\Sigma $ in a Riemannian manifold $(M, g)$ is defined as $H_g = \Div_g \nu$, where $\nu$ is the unit normal vector of $\Sigma$. When $(M, g)$ is asymptotically flat, we choose $\nu$ to point to infinity (and thus the unit normal for $\Sigma$ in $(\Omega, g_0)$ points outward). 

We shall refer to the geometric boundary data $(g_0^\intercal, H_{g_0})$ as the \emph{Bartnik boundary data}.   Let us also remark on the assumption that  $H_{g_0}$ is positive somewhere. The conjecture would fail without this assumption because such extension, if exist, would contain a minimal hypersurface homologous to the boundary (at least in dimensions $n\le 7$), and the extension must be Schwarzschild by Uniqueness Theorem, which put strong restriction on $g_0^\intercal$. See P. Miao~\cite{Miao:2005} for $n=3$. For $n\le 7$, by minimal surface theory, there is an outermost minimal hypersurface homologous to the boundary. By the result of D. Martin, Miao, and the second author \cite[Theorem 1]{Huang-Martin-Miao:2018}, the static potential $u=0$ on the outmost minimal hypersurface. From there, one applies the generalization of Uniqueness of Static Black holes in higher dimensions by G. Gibbons, D. Ida, and T. Shiromizu~\cite{Gibbons-Ida-Shiromizu:2002}. 
  
 Even with the mean curvature assumption, it is highly speculated that Conjecture~\ref{static} does not hold in general as stated.  Let $\Omega$ be a bounded open subset in $\mathbb R^n$. Observed by \cite{Anderson-Khuri:2013, Anderson-Jauregui:2016}, if the boundary $\Sigma=\partial \overline{\Omega}$  is only \emph{inner} embedded, i.e., $\Sigma$ touches itself from the exterior region $\mathbb{R}^n\setminus \Omega$, the induced data $(\bar{g}^\intercal, H_{\bar{g}})$ is valid Bartnik boundary data, but $(\bar{g}, 1)$ in $\mathbb{R}^n\setminus \Omega$ is \emph{not} a valid static vacuum extension as $\mathbb{R}^n\setminus \Omega$ is not a manifold with boundary. One can further arrange so that the mean curvature $H_{\bar g}$ is positive everywhere. Those inner embedded hypersurfaces are conjectured to be counter-examples to Conjecture~\ref{static} by \cite[Conjecture 5.2]{Anderson-Jauregui:2016} (see also \cite{Anderson:2019}), though it is not clear whether there could be another static vacuum extension far way from $(\mathbb R^n\setminus \Omega, \bar g, 1)$. Nevertheless, positive results to  Conjecture~\ref{static}, under suitable assumptions, will provide a  structure theory for the space of static vacuum metrics (parametrized by their Bartnik boundary data). It also connects the fundamental problem on isometric embeddings of hypersurfaces into a static vacuum manifold with prescribed mean curvature. In particular, that question apparently has intriguing connections to the work of P.-N. Chen, M.-T. Wang, Y.-K. Wang, and S.-T. Yau~ \cite{Chen-Wang-Wang-Yau:2018} where the notion of quasi-local energy defined via isometric embeddings into a reference static metric is proposed, extending the celebrated Wang-Yau quasi-local mass with respect to the Minkowski spacetime~\cite{Wang-Yau:2009}.

There are some positive results toward Conjecture~\ref{static}.  The existence and local uniqueness is proven for $n=3$ and for $(g_0, H_{g_0})$  sufficiently close to the induced Bartnik boundary data on a round sphere from the Euclidean metric, i.e. $(g_0, H_{g_0})$ sufficiently close to $(g_{S^2}, 2)$. See Miao~\cite{Miao:2003}, Anderson-Khuri \cite{Anderson-Khuri:2013}, and  Anderson~ \cite{Anderson:2015}. In recent work~\cite{An-Huang:2021}, we give a general framework to tackle Conjecture~\ref{static} and  confirm existence and local uniqueness of Conjecture~\ref{static} for large classes of boundary data, including those close to the induced boundary data on either any star-shaped hypersurfaces or quite general perturbed hypersurfaces in the Euclidean space. In this present note, we improve Theorem 7 in \cite{An-Huang:2021}, by employing new arguments in our recent work \cite{An-Huang:2022}. The new results are presented as  Theorem~\ref{th:one-sided}, Corollary~\ref{co}, and Theorem~\ref{th:trivial} below.

To describe the new results, we first recall the basic notations and definitions and review relevant results from \cite{An-Huang:2021}. 

  Let  $\Omega$ be a bounded open subset in $\mathbb{R}^n$ whose boundary $\Sigma = \partial \overline{\Omega}$ is a connected, embedded smooth hypersurface in $\mathbb{R}^n$.  We denote by $\bar{g}$ the Euclidean metric in $\mathbb{R}^n$ with $\bar g_{ij} = \delta_{ij}$ (with respect to a fixed Cartesian coordinate chart). Our analytic framework is based on the weighted H\"older spaces $\C^{k,\alpha}_{-q} (\mathbb R^n\setminus \Omega)$ (see its definition in Section 2.1 of \cite{An-Huang:2021}), and we always assume the H\"older exponent $\alpha\in(0,1)$ and the fall-off rate $q\in (\frac{n-2}{2} , n-2)$ for asymptotical flatness. We denote by $D\Ric|_{\bar{g}} (h)$ the linearization of the Ricci curvature at $\bar{g}$; namely, let $g(t)$ be an arbitrary family of Riemannian metrics on $\mathbb{R}^n\setminus \Omega$ so that $g(0) = \bar g$ and $g'(0) = h$, then $D\Ric|_{\bar{g}}(h) := \left.\frac{d}{dt}\right|_{t=0} \Ric_{g(t)}$.  Similarly, we define the linearizations of   the mean curvature and second fundamental form on $\Sigma$ by  $DH|_{\bar{g}}(h)$ and  $DA|_{\bar{g}}(h)$, respectively. We will omit the subscript $\bar g$ in those linearizations when the context is clear. 

\begin{Definition}\label{definition:static-regular}
The boundary $\Sigma$ is said to be \emph{static regular in  $\mathbb R^n\setminus\Omega$} if  for any  pair of a symmetric $(0,2)$-tensor $h$ and a scalar-valued function $v$ satisfying $(h, v)\in  \C^{2,\alpha}_{-q}(\mathbb{R}^n\setminus \Omega)$ and
\begin{align}\label{eq:static-bdry}
\begin{split}
-D\Ric (h)+ \nabla^2 v&=0,  \quad  \Delta v=0 \quad \mbox{ in } \mathbb{R}^n\setminus \Omega,\\
h^\intercal&=0, \quad DH(h)=0\quad \mbox{ on } \Sigma,
\end{split}
\end{align}
we must have $DA(h)=0$ on~$\Sigma$. 
\end{Definition}

The following fundamental result obtained  in \cite{An-Huang:2021} says that ``static regular'' is a sufficient condition for existence and local uniqueness. 
\begin{Theorem}[{\cite[Theorem 3]{An-Huang:2021}}]\label{th:static}
Suppose the boundary $\Sigma$ is static regular in $\mathbb{R}^n\setminus \Omega$. Then there exist positive constants $\epsilon_0, C$ such that for each $\epsilon \in (0, \epsilon_0)$, if $(\tau, \phi)$ satisfies~$\| (\tau,\phi)- (\bar{g}^\intercal, H_{\bar{g}}) \|_{\C^{2,\alpha}(\Sigma)\times \C^{1,\alpha}(\Sigma)}<\epsilon$, then there exists an asymptotically flat  pair $(g, u)$ with $\|(g, u) - (\bar{g}, 1) \|_{\C^{2,\alpha}_{-q}(\mathbb{R}^n \setminus \Omega)} < C\epsilon$ such that $(g, u)$ is a static vacuum pair in $\mathbb{R}^n\setminus \Omega$ having the Bartnik boundary data $(g^\intercal, H_g) = (\tau, \phi)$ on~$\Sigma$.

Furthermore, the solution $(g, u)$ is geometrically unique in a neighborhood $\mathcal U$ of $(\bar g, 1)$ in the $\C^{2,\alpha}_{-q}(\mathbb{R}^n\setminus \Omega)$-norm.
\end{Theorem}
We remark that the ``local uniqueness'' in the above theorem is precisely described  under  the static-harmonic gauge and the orthogonal gauge. Since we will not explicitly use them in the present paper, we refer the reader to the discussion right after Theorem~3 in \cite{An-Huang:2021}.

In \cite{An-Huang:2021}, we furthermore show that large classes of hypersurfaces in $\mathbb R^n$ are static regular. %We begin with convex surfaces in $\mathbb{R}^3$.
%\begin{Theorem}[{\cite[Theorem 6]{An-Huang:2021}}]
%Let  $\Omega$ be a bounded open subset in $\mathbb{R}^3$ whose boundary $\Sigma=\partial \overline{\Omega}$ has positive Gauss curvature. Then $\Sigma$ is static regular in~$\mathbb{R}^3\setminus \Omega$.
%\end{Theorem}
In particular, we show that static regular hypersurfaces are ``dense'' in the following concrete sense. A family of embedded hypersurfaces $\{ \Sigma_t\}\subset \mathbb{R}^n$ is said to form a \emph{smooth generalized foliation} if the deformation vector $X$ of $\{ \Sigma_t\}$ is smooth and on each $\Sigma_t$, $\bar g( X, \nu) =\zeta$ where $\zeta>0$  in a dense subset of $\Sigma_t$, and $\nu$ is the unit normal of $\Sigma_t$. In other words, $\{ \Sigma_t\}$ is slightly more general than a foliation in that the leaves can overlap on a nowhere dense subset.

\begin{Theorem}[{\cite[Theorem 7]{An-Huang:2021}}]\label{th:perturbation}
Let $\delta>0$, $t\in [-\delta, \delta]$, and each $\Omega_t\subset \mathbb{R}^n$ be a bounded open subset with hypersurface boundary $\Sigma_t = \partial\overline{ \Omega_t}$  embedded  in $\mathbb{R}^n$.  Suppose  the boundaries $\{ \Sigma_t\}$ form a smooth generalized foliation. Then  there is an open dense subset $J\subset (-\delta,\delta)$ such that  $\Sigma_t$  is static regular in $\mathbb{R}^n\setminus \Omega_t$ for all $t\in J$.
\end{Theorem}

The above theorem has the following strong consequence because of the dilation property of the Euclidean static vacuum pair.
\begin{Corollary}[{\cite[Corollary 8]{An-Huang:2021}}]
Let  $\Omega$ be a bounded open subset in $\mathbb{R}^n$ whose boundary $\Sigma=\partial \overline{\Omega}$ is a star-shaped hypersurface. Then $\Sigma$ is static regular in $\mathbb{R}^n\setminus \Omega$. 
\end{Corollary}

The purpose of this note is to extend Theorem~\ref{th:perturbation}. 

\begin{Definition}\label{def}
A collection of embedded hypersurfaces $\{ \Sigma_t\}\subset \mathbb{R}^n$ is  a \emph{smooth one-sided family of hypersurfaces foliating along  $\hat{\Sigma}_t \subset \Sigma_t$ with  relatively simply-connected $\Sigma_t\setminus \hat{\Sigma}_t$} if the deformation vector $X$ of $\{ \Sigma_t\}$ is smooth and on each $\Sigma_t$, $\bar g( X, \nu) =\zeta \ge 0$ with $\zeta>0$  in a dense subset of  each $\hat{\Sigma}_t \subset \Sigma_t$ satisfying $\pi_1(\Sigma_t, \hat{\Sigma}_t)=0$. 
\end{Definition}

For a subset $U$ in $M$, the condition $\pi_1(M, U)=0$ says that $U$ is connected and the inclusion map $U \hookrightarrow M$ induces a surjection $\pi_1(U) \to \pi_1(M)$. In our setting, we clearly have $\pi_1(\mathbb R^n \setminus \Omega, \Sigma)=0$. Thus, the additional condition for a subset $\hat{\Sigma}\subset \Sigma$ to satisfy $\pi_1(\Sigma, \hat{\Sigma})=0$ implies $\pi_1(\mathbb R^n\setminus \Omega, \hat{\Sigma})=0$. In loose terms, the later condition implies that for any point $x\in \mathbb R^n\setminus \Omega$, all paths from $x$ to $\hat{\Sigma}$ are (homotopically) equivalent. See Figure~\ref{one-sided} below. This property is used in Theorem~\ref{th:extension} below to ensure certain global extensions of local vector fields.

  \begin{figure}
\begin{center}
\includegraphics[width=0.8\textwidth]{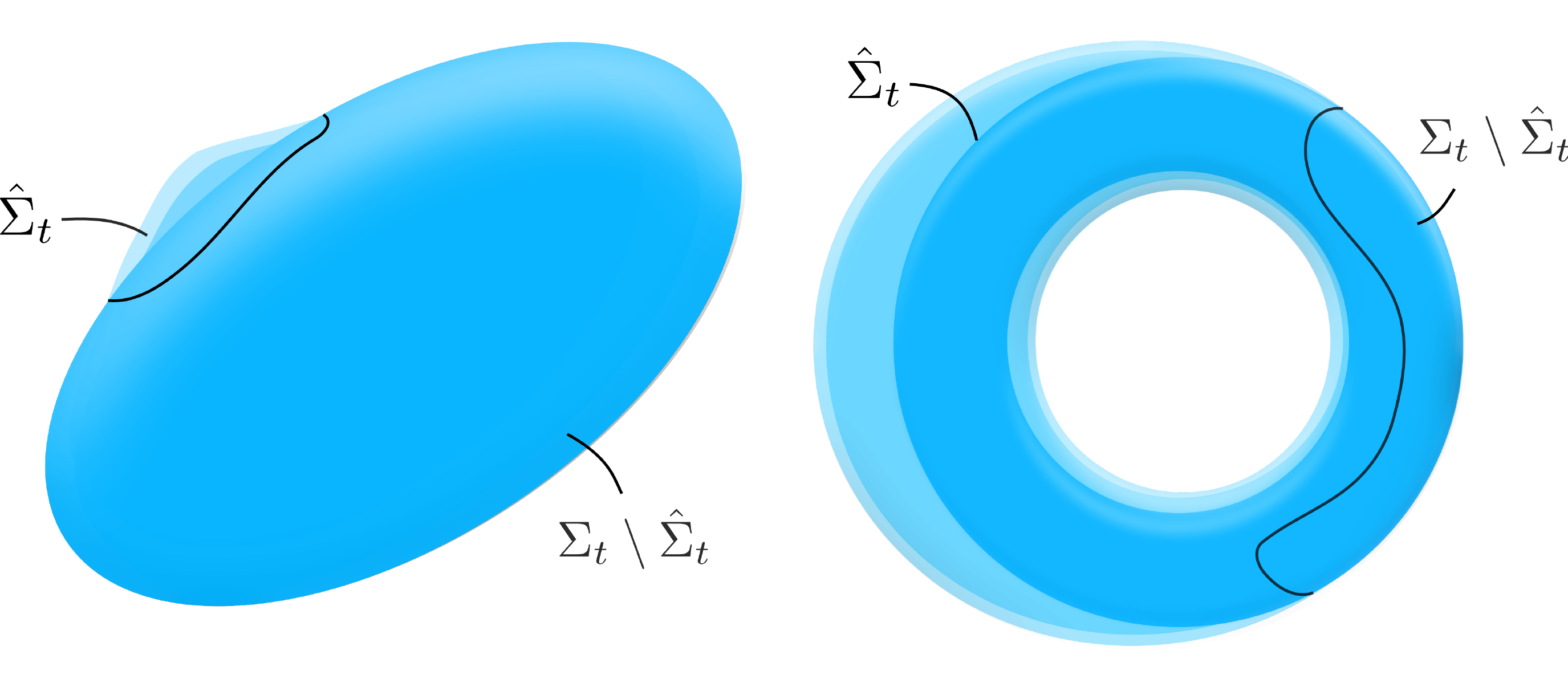}
\caption{Each figure illustrates Definition~\ref{def} that $\{\Sigma_t\}$ foliates along $\hat{\Sigma}_t$  with  relatively simply-connected $\Sigma_t\setminus {\hat \Sigma}_t$. In the left figure, a one-sided family of (topological) spheres $\{\Sigma_t\}$ is shown where $\hat{\Sigma}_t$ can be very small. In other words, $\Sigma_t$ can largely overlap on $\Sigma_t\setminus \hat{\Sigma}_t$. The right figure illustrates a one-sided family of (topological) tori with $\pi_1(\Sigma_t, \hat{\Sigma}_t)=0$.  
}
\label{one-sided}
\end{center}
\end{figure}

Note that a smooth generalized foliation  $\{\Sigma_t\}$ defined earlier is necessarily a smooth one-sided family of hypersurfaces foliating along $\hat{\Sigma}_t \subset \Sigma_t$ with relatively simply-connected $ \Sigma_t\setminus \hat{\Sigma}_t$ (by letting $\hat\Sigma_t = \Sigma_t$). However, a smooth one-sided family $\{ \Sigma_t\}$ in the above sense may not form a foliation because the leaves can overlap on $\Sigma_t \setminus \hat{\Sigma}_t $. The following theorem generalizes Theorem~\ref{th:perturbation}.

\begin{Theorem}\label{th:one-sided}
Let $\delta>0$, $t\in [-\delta, \delta]$, and each $\Omega_t\subset \mathbb{R}^n$ be a bounded open subset with hypersurface boundary $\Sigma_t = \partial\overline{ \Omega_t}$  embedded  in $\mathbb{R}^n$.  Suppose the boundaries $\{ \Sigma_t\}$ form a smooth one-sided family of hypersurfaces foliating along  $\hat{\Sigma}_t \subset \Sigma_t$ with relatively simply-connected $\Sigma_t\setminus \hat{\Sigma}_t$. Then  there is an open dense subset $J\subset (-\delta,\delta)$ such that  $\Sigma_t$  is static regular in $\mathbb{R}^n\setminus \Omega_t$ for all $t\in J$.
\end{Theorem}

In the special case that $\Sigma$ is simply-connected (e.g., $\Sigma$ is a topological sphere), we trivially have $\pi_1(\Sigma, \hat \Sigma)=0$ for any nonempty connected subset $\hat \Sigma$. Slight perturbation on $\hat \Sigma$ produces  a one-sided family $\{\Sigma_t\}$ with $\Sigma_0=\Sigma$ that foliates along small subsets ${\hat \Sigma}_t\subset \Sigma_t$ so that  $\Sigma_t \setminus \hat {\Sigma}_t$ coincides with $\Sigma\setminus \hat \Sigma$ for all $t$. See the left figure in Figure~\ref{one-sided}. Theorem~\ref{th:one-sided} says that $\Sigma_t$ is static regular for $t$ in an open and dense set. Together with Theorem~\ref{th:static}, we give  the following corollary that one can solve for static vacuum extensions whose boundary data are arbitrarily close to the induced boundary data $(\bar g^\intercal, H_{\bar g})$ on $\Sigma$, except a small subset $\hat \Sigma\subset \Sigma$.

\begin{Corollary}\label{co}
Let $\Sigma= \partial \overline{\Omega}$ be a simply-connected, closed, embedded hypersurface in $\mathbb R^n$.  Given any nonempty open subset $\hat \Sigma\subset \Sigma$ and  any $\delta>0$, there exists $(\tau_0, \phi_0)\in \C^{2,\alpha}(\Sigma)\times \C^{1,\alpha}(\Sigma)$ and  constants $\epsilon_0, C>0$ satisfying 
\begin{align*}
	&(\tau_0, \phi_0) = (\bar g^\intercal, H_{\bar g} ) \mbox{ on } \Sigma\setminus \hat \Sigma\\
	&\| (\tau_0, \phi_0) - (\bar g^\intercal, H_{\bar g} )\|_{\C^{2,\alpha}(\hat \Sigma) \times \C^{1,\alpha}(\hat \Sigma)} < \delta
\end{align*}
such that for each $\epsilon \in (0, \epsilon_0)$, if $(\tau, \phi)$ satisfies~$\| (\tau,\phi)- (\tau_0, \phi_0) \|_{\C^{2,\alpha}(\Sigma)\times \C^{1,\alpha}(\Sigma)}<\epsilon$, then there exists an asymptotically flat  pair $(g, u)$ with $\|(g, u) - (\bar{g}, 1) \|_{\C^{2,\alpha}_{-q}(\mathbb{R}^n \setminus \Omega)} < C\epsilon$ such that $(g, u)$ is a static vacuum pair in $\mathbb{R}^n\setminus \Omega$ having the Bartnik boundary data $(g^\intercal, H_g) = (\tau, \phi)$ on~$\Sigma$.

Furthermore, the solution $(g, u)$ is geometrically unique in a neighborhood $\mathcal U$ of $(\bar g, 1)$ in the $\C^{2,\alpha}_{-q}(\mathbb{R}^n\setminus \Omega)$-norm.

\end{Corollary}

The proof to Theorem~\ref{th:one-sided} involves several new arguments used in our recent work for general asymptotically flat, static vacuum background metrics~\cite{An-Huang:2022}. One of the key arguments is the following theorem, which can be viewed as a uniqueness theorem for ``localized'' boundary data. 

\begin{Theorem}\label{th:trivial}
Let $\hat{\Sigma}$ be an open subset of $\Sigma$ (can be the entire $\Sigma$) satisfying $\pi_1(\Sigma, \hat{\Sigma})=0$. Let $(h, v)\in \C^{2,\alpha}_{-q}(\mathbb R^n\setminus \Omega)$ solve 
\begin{align*}
	&\left \{ \begin{array}{l}-D\Ric(h) + \nabla^2 v = 0\\
	\Delta v=0 \end{array} \right.	 \mbox{ in } \mathbb R^n\setminus \Omega\\
	&\left \{ \begin{array}{l} h^\intercal = 0\\
	DA(h)=0 \\
	 D (\nabla_\nu A)  (h) = 0\end{array} \right. \quad \mbox{ on }\hat{\Sigma}
\end{align*}
where $ D (\nabla_\nu A)  (h)$ denotes the linearization of $\nabla_\nu A$. Then there is a vector field $X\in \C^{3,\alpha}_{\mathrm{loc}}(\mathbb R^n\setminus \Omega) $ satisfying $X=0$ on $\hat \Sigma$ and $X-K \in \C^{3,\alpha}_{-q}(\mathbb R^n\setminus \Omega)$ for some Euclidean Killing vector field $K$ (possibly zero) such that 
\[
	 h = L_X \bar g\quad \mbox{ and } \quad v= 0\quad  \mbox{ in } \mathbb R^n\setminus \Omega.
\]
Furthermore, if $h^\intercal =0$ and $DH(h)=0$ everywhere on $\Sigma$, then $X=0$ everywhere on $\Sigma$ and thus $DA(h)=0$ on $\Sigma$. 
\end{Theorem}

Theorem~\ref{th:trivial} says that the solutions  must be ``trivial'' in the sense that $(h, v)$ must arise from ``infinitesimal'' diffeomorphisms. More precisely, if we let $X$ be a vector field as in the above theorem and let $\phi_t$ be the family of diffeomorphisms on $\mathbb R^n\setminus \Omega$ generated from $X$ (in particular, $\phi_t$ is the identity map on $\Sigma$ for all $t$). Then the family of static vacuum pairs $(g_t, u_t)= \phi_t^*( \bar g,  1)$ as the pull-back pairs of $(\bar g, 1)$ would also satisfy \eqref{eq:static} and have the same boundary data $(\bar g^\intercal, A_{\bar g}, \nabla_{\nu } A_{\bar g} )$ on~$\hat{\Sigma}$ (in fact on the entire $\Sigma$).\footnote{Here $\nabla_{\nu_g } A_{g}$ means the $\nu_g$-covariant derivative of the second fundamental forms of $g$-equidistant hypersurfaces to the boundary $\Sigma$.} The linearization of $(g_t, u_t)$ becomes $(L_X \bar g, X(1)) = (L_X \bar g, 0)$, which satisfies the linearized system in Theorem~\ref{th:trivial}. On the other hand, Theorem~\ref{th:trivial} says that those are the \emph{only} solutions.

The rest of this note is organized as  follows:  Theorem~\ref{th:trivial} is proved in Section~\ref{se:boundary}, and then Theorem~\ref{th:one-sided} is proved in Section~\ref{se:perturbed}.

\section{Localized boundary data}\label{se:boundary}

%We denote the static vacuum operator, corresponding to the static vacuum equation~\ref{eq:static} by 
%\[
%	S(g, u) = ( -u \Ric_g + \nabla^2_g u , \, \Delta_g u). 
%\]
%Then we let $DS(h, v)$ to denote the linearization of $S(g, u)$ at the Euclidean pair $(\bar g, 1)$:
%\[
%	DS(h, v) := (- D\Ric(h) + \nabla^2 v, \, \Delta v)	
%\]	
%where the covariant derivatives are all taken with respect to the Euclidean metric $\bar g$.  

The major motivation for the definition of static regular, Definition~\ref{definition:static-regular}, is the following uniqueness theorem for Cauchy boundary data from \cite{An-Huang:2021}.

\begin{theorem}
Let $(h, v)\in \C^{2,\alpha}_{-q}(\mathbb R^n\setminus \Omega)$ solve 
\begin{align*}
	&\left \{ \begin{array}{l}-D\Ric(h) + \nabla^2 v = 0\\
	\Delta v=0 \end{array} \right.	 \mbox{ in } \mathbb R^n\setminus \Omega\\
	&\left \{ \begin{array}{l} h^\intercal = 0\\
	DA(h)=0\end{array} \right. \quad \mbox{ on } \Sigma.
\end{align*}
Then there is a vector field $X\in \C^{3,\alpha}_{\mathrm{loc}}(\mathbb R^n\setminus \Omega) $ satisfying $X=0$ on $\Sigma$ and $X-K \in \C^{3,\alpha}_{1-q}(\mathbb R^n\setminus \Omega)$ for some Euclidean Killing vector field $K$ (possibly zero) such that 
\[
	 h = L_X \bar g\quad \mbox{ and } \quad v= 0 \mbox{ in } \mathbb R^n\setminus \Omega.
\]
\end{theorem}
\begin{proof}
From the proof of Lemma 4.8 in \cite{An-Huang:2021}, we see that $v=0$ in $\mathbb R^n\setminus \Omega$. Therefore, $h$ is a Ricci flat deformation in the sense that $D\Ric(h)=0$ in $\mathbb R^n\setminus \Omega$. Then by Theorem 2.8 of \cite{An-Huang:2021}, we get the desired conclusion. 
\end{proof}

The goal of this section is to prove  Theorem~\ref{th:trivial}, whose main difference from the above theorem is  that the boundary conditions for Theorem~\ref{th:trivial}  are ``localized'' only on a subset $\hat{\Sigma}\subset \Sigma$ satisfying  $\pi_1(\Sigma, \hat{\Sigma}) = 0$. 

We will first establish some basic results, and then Theorem~\ref{th:trivial} follows immediately after proving Proposition~\ref{pr:v} and Proposition~\ref{pr:Ricci} below.

We say that a symmetric $(0, 2)$-tensor $h$ is said to satisfy the  \emph{geodesic gauge (of order $2$)} on $\Sigma$ if 
\[
	h(\nu, \cdot )=0, \quad  (\nabla_\nu h)(\nu, \cdot)=0 , \quad  (\nabla_\nu^2 h)(\nu, \cdot)=0 \quad \mbox{ on }  \Sigma
\]
 where $\nu$ is the unit normal vector of $\Sigma$ parallelly extended into a collar neighborhood of $\Sigma$.  
 
 Following the same argument as in \cite[Lemma 2.5]{An-Huang:2021}, we see that any tensor $h$ can be ``transformed'' to satisfy the geodesic gauge. 
 \begin{lemma}[Cf. {\cite[Lemma 2.5]{An-Huang:2021}}]\label{le:geodesic}
 Let $h\in \C^{2,\alpha}_{\mathrm{loc}}(\mathbb R\setminus \Omega)$ be a symmetric $(0,2)$-tensor. Then there exists a vector field $V\in \C^{3,\alpha}$ with $V=0$ on $\Sigma$ and $V$ vanishing outside a collar neighborhood of $\Sigma$ such that $k:= h+L_V \bar g$ satisfies the geodesic gauge on $\Sigma$. 
 \end{lemma}

The following lemma gives an analytic interpretation for the geometric boundary conditions of Theorem~\ref{th:trivial}.

\begin{lemma}\label{le:zero}
Let $\hat{\Sigma}$ be an open subset of the boundary $\Sigma$ (can be the entire $\Sigma$). Suppose $h \in \C^{2,\alpha}_{\mathrm{loc}}(\mathbb {R}^n\setminus \Omega)$ satisfies the geodesic gauge  and 
\[
h^\intercal =0, \quad DA(h)=0, \quad D (\nabla_\nu A)(h)=0 \quad \mbox{ on }  \hat{\Sigma}.   
\]
Then 
\[
	h=0, \quad \nabla h = 0, \quad \nabla^2 h =0 \quad \mbox{ on } \hat{\Sigma}. 
\]
\end{lemma}
\begin{proof}
The first identity is an immediate consequence of $h^\intercal=0$ and the geodesic gauge. 

To show the second identity, it suffices to show $(\nabla_\nu h)^\intercal=0$ because $(\nabla_{\mathrm{arbitrary}} h)(\nu,\cdot)=0$  and $(\nabla_{\mathrm{tangential}} h)^\intercal=0$ from $h=0$ and geodesic gauge. Recall the formula (see \cite[Equation (2.3)]{An-Huang:2021})
\begin{align}\label{eq:sff}
	DA(h)&= \tfrac{1}{2} (\nabla_\nu h)^\intercal	+ A\circ h - \tfrac{1}{2} L_\omega g^\intercal - \tfrac{1}{2} h (\nu, \nu) A
\end{align}
where the one-form $\omega$ is defined by $\omega(\cdot) = h(\nu, \cdot)$, $A$ is the second fundamental form of $\Sigma\subset (\mathbb R^n, \bar g)$, and $(A\circ h)_{ab}= \tfrac{1}{2} (A_{ac} h_b^c + A_{bc} h^c_a)$. Therefore, the assumption $DA(h)=0$ implies that $(\nabla_\nu h)^\intercal=0$ and thus $\nabla h=0$. 

To show $\nabla^2 h =0$, we just need to show that $(\nabla_\nu^2 h)^\intercal =0$ because $\nabla_{\mathrm{tangential}} (\nabla h)=0$ and $\nabla_\nu \nabla_{\mathrm{tangential}} h= \nabla_{\mathrm{tangential}} \nabla_\nu h$ plus terms involving $h$ and $\nabla h$, which are all zero on $\hat \Sigma$.  Note that since $h$ satisfies the geodesic gauge, we also have   $\nabla_\nu (DA)(h)= D(\nabla_\nu A)(h)$ on $\hat \Sigma$. To see this, we compute, for tangential vectors $e_a, e_b$ to $\Sigma$:
\begin{align}\label{eq:normal}
\begin{split}
	\Big(D(\nabla_\nu A)(h)\Big)(e_a, e_b) &:=\left. \frac{d}{dt}\right|_{t=0} \big( \nabla_{\nu_{g(t)}} A_{g(t)} \big) (e_a, e_b)\\
	&=\Big(\nabla_\nu (DA)(h)\Big) (e_a, e_b) +\big((\nabla_{\nu_{g(t)}})' A\big)(e_a, e_b)\\
	&=\Big(\nabla_\nu (DA)(h)\Big) (e_a, e_b)
\end{split}
\end{align}
where in the the second line $(\nabla_{\nu_{g(t)}})' := \left. \frac{d}{dt}\right|_{t=0}\nabla_{\nu_{g(t)}}$ and one can verify that $\big((\nabla_{\nu_{g(t)}})' A\big)(e_a, e_b)$ because $h=0$ and $\nabla h=0$ on $\hat \Sigma$.

To conclude, we get $\nabla_\nu (DA)(h)= 0$ on $\hat \Sigma$  using \eqref{eq:normal} and the assumption that $D \big( \nabla_\nu A\big)(h) =0$ on $\hat \Sigma$. Covariant differentiating \eqref{eq:sff} in $\nu$, we obtain $(\nabla_\nu^2 h)^\intercal =0$.

\end{proof}

\begin{proposition}\label{pr:v}
Let $\hat{\Sigma}$ be an open subset of $\Sigma$ (can be the entire $\Sigma$). Let $(h, v)\in \C^{2,\alpha}_{-q}(\mathbb R^n\setminus \Omega)$ solve 
\begin{align*}
	&\left \{ \begin{array}{l}-D\Ric(h) + \nabla^2 v = 0\\
	\Delta v=0 \end{array} \right.	 \mbox{ in } \mathbb R^n\setminus \Omega\\
	&\left \{ \begin{array}{l} h^\intercal = 0\\
	DA(h)=0 \\
	D \big( \nabla_\nu A \big) (h) = 0\end{array} \right. \quad \mbox{ on } \hat{\Sigma}.
\end{align*}
Then $v\equiv 0$ and $D\Ric(h)=0$ in $\mathbb R^n\setminus \Omega$.
\end{proposition}
\begin{proof}
By Lemma~\ref{le:geodesic}, we may assume $h$ satisfies the geodesic gauge. The boundary conditions and  Lemma~\ref{le:zero} imply that $h=0, \nabla h=0, \nabla^2 h =0$ on $\hat{\Sigma}$. 

Note that $D\Ric(h)$ involves $h$ and its derivatives up to the second order; precisely,  in local coordinates (see, e.g. \cite[Equation (2.1)]{An-Huang:2021}):
\begin{align*}
	(D\Ric|_g(h))_{ij} &=-\tfrac{1}{2} g^{k\ell} h_{ij;k\ell} + \tfrac{1}{2} g^{k\ell} (h_{i k; \ell j} + h_{jk; \ell i} ) - \tfrac{1}{2} (\tr h)_{ij}  \\
	&\quad + \tfrac{1}{2} (R_{i\ell} h^\ell_j + R_{j\ell} h^\ell_i )- R_{ik\ell j} h^{k\ell} \quad \mbox{ for all } i, j =0,1,\dots, n-1.
\end{align*}
Restricting $-D\Ric(h)+\nabla^2 v=0$ on the boundary, we see $\nabla^2 v=0 $ on $\hat{\Sigma}$.

Next, we define the function $f=\frac{\partial v}{\partial x_i }$ for some $i=1,\dots, n$, where $(x_1, \dots, x_n)$ are the Cartesian coordinates.  Since $v$ is harmonic, $f$ is also a harmonic function in $\mathbb R^n\setminus \Omega$. The conditions that $\nabla^2 v=0$ on $\hat{\Sigma}$ imply 
\[
	\nabla^\Sigma f =0\quad \mbox{ and } \quad \nu (f) =0 \qquad \mbox{ on } \hat{\Sigma}.
\]
The first identity implies $f\equiv c$ for some constant $c$ in a connected open subset of $\hat{\Sigma}$. Together with the second identity and uniqueness of Cauchy boundary value for the harmonic equation, we conclude that $f\equiv c$ everywhere in $\mathbb R^n\setminus \Omega$. Since $f\to 0$ by the fall-off rate of $v$, we see that $f=\frac{\partial v}{\partial x_i}$ is identically zero. Since $i$ is arbitrary, we see that $\nabla v=0$  and thus $v$ is constant. Since  $v\to 0 $ at infinity, we conclude that $v$ is identically zero. 
\end{proof}

In the above proposition, we have shown $D\Ric(h)=0$ in $\mathbb {R}^n\setminus \Omega$. Proposition~\ref{pr:Ricci} below generalizes Theorem 2.8 of \cite{An-Huang:2021}, where $\hat{\Sigma}$ was assumed to be the entire boundary $\Sigma$. The key argument is the following extension theorem for $h$-Killing vector fields, that extends the classical result of Nomizu for the case $h$ is identically zero~\cite{Nomizu:1960}.  
\begin{theorem}[{\cite[Theorem 7]{An-Huang:2022}, Cf. \cite[Lemma 2.6]{Anderson:2008}}]\label{th:extension}
Let $(M, g)$ be a connected, analytic Riemannian manifold. Let $h$ be an analytic, symmetric $(0, 2)$-tensor on $M$. Let $U\subset M$ be a connected open subset satisfying $\pi_1(M, U)=0$. Then if $h =L_X g $ in $U$, there is a unique global vector field $Y$ such that $Y= X$ in $U$ and $h=L_{Y} g$ in the whole manifold $M$.
\end{theorem}

\begin{proposition}\label{pr:Ricci}
Let $\hat{\Sigma}$ be an open subset of $\Sigma$ satisfying $\pi_1(\Sigma, \hat{\Sigma})=0$. Let $h\in \C^{2,\alpha}_{-q}(\mathbb R^n\setminus \Omega)$ satisfy
\begin{align*}
	&D\Ric(h)  = 0 \quad 	 \mbox{ in } \mathbb R^n\setminus \Omega\\
	&\left \{ \begin{array}{l} h^\intercal = 0\\
	DA(h)=0 \end{array} \right. \quad \mbox{ on } \hat{\Sigma}.
\end{align*}
Then there is a vector field $X\in \C^{3,\alpha}_{\mathrm{loc}}(\mathbb R^n\setminus \Omega) $ satisfying $X=0$ on $\hat \Sigma$ and $X-K \in \C^{3,\alpha}_{-q}(\mathbb R^n\setminus \Omega)$ for some Euclidean Killing vector field $K$ (possibly zero) such that 
\[
	 h = L_X \bar g  \mbox{ in } \mathbb R^n\setminus \Omega.
\]
Furthermore, if $h^\intercal =0$ and $DH(h)=0$ everywhere on $\Sigma$, then $X=0$ everywhere on $\Sigma$ and thus $DA(h)=0$ on $\Sigma$. 
\end{proposition}
\begin{proof}
We may without loss of generality assume that $h$ satisfies the geodesic gauge on $\Sigma$. We extend $h$ by $0$ across $\hat{\Sigma}$ into some small open subset $U\subset \Omega$ so that the ``extended'' manifold $\hat M = (\mathbb R^n \setminus \Omega)\cup \overline{U}$ has smooth embedded boundary $\partial \hat M$ and $\pi_1(\hat M, U)=0$. Denote the extension of $h$ by~$k\in \C^1_{\mathrm{loc}}(\hat M)$:
\[
 	k = \left\{ \begin{array}{ll} h  & \mbox{ in } \mathbb R^n\setminus \Omega \\
	0 &\mbox{ in } U\end{array} \right.
\]	 
Let $Z\in \C^{1,\alpha}_{1-q}(\hat M)$ be a vector field that weakly solves $\Delta Z = \beta k$ in $\hat M$ with $Z=0$ on $\partial \hat M$ where the Bianchi operator $\beta k = -\Div_{\bar g} k + \frac{1}{2} d(\mathrm{tr}_{\bar g} k)$. Or equivalently, $k+L_Z \bar g$ weakly solves $\beta (k+L_Z \bar g)  =0$ in $\hat M$. Together with the assumption that $D\Ric(h)=0$ in $\mathbb R^n\setminus \Omega$ and the boundary condition $h=0, \nabla h=0$ on $\Sigma$, we have that $k+L_Z \bar g$ is a weak solution to $\Delta (k+L_Z \bar g )=0$ in $\hat M$. So far, the argument has followed closely \cite[Theorem 2.8]{An-Huang:2021}, to which we refer the analytic details. 

But in the current setting we cannot conclude $k+L_Z\bar g$ is identically zero as in \cite[Theorem 2.8]{An-Huang:2021}. (In \cite{An-Huang:2021}, it was possible to extend the harmonic $k+L_X\bar g$ globally on the entire $\mathbb R^n$.) Here we apply Weyl's lemma to see that $k+L_Z \bar g$ is analytic in $\Int (\hat M)$. Since $k+L_Z \bar g=L_Z \bar g$ in $U$ (remember $k\equiv 0$ there), by Theorem~\ref{th:extension}, there is a unique vector field $Y$ such that $Y=Z$ in $U$ and 
\[
	k +L_Z \bar g = L_Y \bar g \mbox{ in } \hat M.
\]
To summarize, we obtain $X = Y-Z$ with   $X= 0 $ on $\hat{\Sigma}$ and 
\[
	h = L_X\bar g \quad \mbox{ in } \mathbb R^n\setminus \Omega.
\]
Note that $X\in \C^{3,\alpha}_{\mathrm{loc}}(\mathbb R\setminus \Omega)$ because of the regularity $h$.

The rest of the conclusions follow from basic arguments as in \cite{An-Huang:2022}, so we just give a sketch below.  To show the desired asymptotic of $X$ toward infinity, one first considers the ODE for $X$ along any ray to  infinity to show that $X= o(|x|^2)$. Then writing the equation $D\Ric(L_X\bar g)=0$ in the harmonic gauge gives a harmonic expansion for $X$. Thus $X$ is asymptotic to a Euclidean Killing vector field $K$, using the fall-off rate $L_X \bar g = h\in \C^{2,\alpha}_{-q}$. 

Lastly, to show that $X=0$ on $\Sigma$ under the added assumptions $h^\intercal=0$ and $DH(h)=0$ on $\Sigma$, we write $X = \eta \nu + X^\intercal$, where $X^\intercal$ is tangential to $\Sigma$. The assumptions $h^\intercal =0$ and $DH(h)=0$ on $\Sigma$ imply $\eta, X^\intercal$ satisfies a  linear PDE system on $\Sigma$. Since $\eta, X^\intercal$ are identically zero on $\hat \Sigma$, by unique continuation, they are identically zero everywhere on $\Sigma$. 

\end{proof}

\begin{proof}[Proof of Theorem~\ref{th:trivial}]
Let $(h, v)$ be as in the statement of Theorem~\ref{th:trivial}. By Proposition~\ref{pr:v}, $v\equiv 0$ in $\mathbb R^n\setminus \Omega$, and thus $h$ satisfies the assumptions in Proposition~\ref{pr:Ricci}, which implies the desired conclusion. 
\end{proof}

\section{A smooth one-sided family of hypersurfaces}\label{se:perturbed}

Let $\delta>0$ and let $\Omega_t\subset \mathbb R^n$, $t\in [-\delta, \delta]$, be bounded open subsets such that their boundaries $\Sigma_t$ are connected, embedded hypersurfaces and $\{ \Sigma_t\}$ form a smooth one-sided family foliating along $\hat {\Sigma}_t\subset \Sigma_t$ with relatively simply-connected $\Sigma_t\setminus \hat{\Sigma}_t$. Namely, their deformation vector $X$ is smooth and on each $\Sigma_t$, $\bar g( X, \nu) =\zeta \ge 0$ with $\zeta>0$ in a dense subset  of $\hat{\Sigma}_t \subset \Sigma_t$ satisfying $\pi_1(\Sigma_t, \hat{\Sigma}_t)=0$. Let $\psi_t: \mathbb R^n\setminus \Omega_t \to \mathbb R^n$ be the flow of $X$. Let us denote $\Omega = \Omega_0,  \Sigma = \Sigma_0, $ and $\hat{\Sigma} = \hat{\Sigma}_0$. Then $\Omega_t = \psi_t(\Omega),  \Sigma_t = \psi_t (\Sigma)$. Denote by $g_t = \psi_t^* (\bar  g|_{\mathbb R^n\setminus \Omega_t})$ the pull-back metric defined on $\mathbb {R}^n\setminus \Omega$. We also note $g_0 = \bar g$.  

Let us define a family of linear operators, with respect to $g_t$, as 
\[
L_t: \C^{2,\alpha}_{-q}(\mathbb R^n\setminus \Omega) \to \C^{0,\alpha}_{-q-2} (\mathbb R^n\setminus \Omega) \times \mathcal B(\Sigma)
\]
\begin{align*}
	L_t(h, v) =\begin{array}{l}\left\{ \begin{array}{ll} -D\Ric|_{g_t}(h) + \nabla^2_{g_t} v \\
	\Delta_{g_t} v \end{array} \right. \quad \mbox{ in } \mathbb R^n\setminus \Omega\\
	\left\{ \begin{array}{ll} h^\intercal \\ 
	DH|_{g_t}(h) \end{array}\right.\quad \mbox{ on } \Sigma
	\end{array}.
\end{align*}
Here, $\mathcal B(\Sigma)=  \C^{2,\alpha}(\Sigma)\times \C^{1,\alpha}(\Sigma)$ is the function space for the boundary operator. Note that each $L_t $ is the pull-back of the operator corresponding to the boundary value problem \eqref{eq:static-bdry} in $\mathbb R^n\setminus \Omega_t$.

In \cite{An-Huang:2021}, we observed that the kernel spaces $\Ker L_t$ have the following properties. 
\begin{proposition}[Cf. {\cite[Proposition 6.6]{An-Huang:2021}}]\label{pr:limit}
There is an open dense subset $J\subset (-\delta, \delta)$ such that for every $a\in J$ and every $(h, v) \in \Ker L_a$, there is a sequence $\{ t_j\}$ in $J$ such that $t_j\searrow a$, $(h(t_j), v(t_j))\in \Ker L_{t_j}$ and $(p, z)\in \C^{2,\alpha}_{-q}(\mathbb R^n\setminus \Omega)$ such that, as $t_j \searrow a$, 
\begin{align*}
	\big( h(t_j), v(t_j) \big) &\to (h, v)\\
	\frac{\big( h(t_j), v(t_j) \big) - (h,v) }{t_j - a} &\to (p, z)
\end{align*}
where both convergence are taken in the $\C^{2,\alpha}_{-q}(\mathbb R^n\setminus \Omega)$-norm. 
\end{proposition}

\begin{remark}
In \cite[Proposition 6.6]{An-Huang:2021}, we actually proved the statement for the kernel of the corresponding ``gauged'' operators, which extends directly  to the above statement. 
\end{remark}

\begin{theorem}\label{th:normal}
Let $J\subset (-\delta,\delta)$ be the open dense subset as in Proposition~\ref{pr:limit}. Then for ever $a\in J$ and every $(h, v)\in \Ker L_a$, ,we have 
\[
	DA|_{g_a}(h) = 0 \quad \mbox{ and } \quad D(\nabla_\nu A)|_{g_a}(h) = 0 \qquad \mbox{ on }\quad \Sigma^{+}_a
\]
where $\Sigma^+_a= \{ x\in \Sigma: \psi_a^*( \zeta|_{\Sigma_a} )(x)>0\}$. (In other words, $\psi_a(\Sigma^+_a)$ is the subset of $\Sigma_a$ on which $\zeta>0$.)
\end{theorem}
\begin{proof}
By re-paramerizing, we may assume $a=0$ and hence $g_a = \bar g, \Sigma_a = \Sigma$, and we denote by $L_a = L$ and $\Sigma_a^+= \Sigma^+$. We may also without loss of generality assume that $h$ satisfies the geodesic gauge. As proven in \cite[Theorem 7 and Theorem 7$^\prime$]{An-Huang:2021},  $\big(p-L_X h, z - X(v)\big)$ is a static vacuum deformation in $\mathbb R^n\setminus \Omega$ satisfying the boundary conditions on $\Sigma$:
\begin{align}\label{eq:boundary}
\begin{split}
	(p-L_X h )^\intercal &= -2\zeta DA(h)\\
	DH(p-L_X h) &=\zeta  A\cdot DA(h).
\end{split}
\end{align}
 
Recall a consequence of the Green-type identity from \cite[Corollary 3.5]{An-Huang:2021}: If both $(h, v), (k, w)\in \C^{2,\alpha}_{-q}(\mathbb R^n\setminus \Omega)$ are static vacuum deformations at $(\bar g,1)$ and $h$ satisfies $h^\intercal=0, DH(h)=0$ on $\Sigma$, then 
\begin{align*}
	\int_\Sigma  \Big\langle \big(vA + DA(h) -\nu(v) \bar{g}^\intercal, 2v \big), \big(k^\intercal, DH(k)\big) \Big\rangle \, \da=0.
\end{align*}
We apply the previous identity by substituting $(k, w) := \big(p-L_X h, z - X(v)\big)$ and using the boundary conditions \eqref{eq:boundary} to obtain 
\begin{align*}
	0&=\int_\Sigma  \Big\langle \big(vA + DA(h) -\nu(v) \bar{g}^\intercal, 2v \big), \big(-2\zeta DA(h), \zeta A\cdot DA(h) \big) \Big\rangle \, \da\\
	&=- \int_\Sigma 2\zeta |DA(h)|^2 \, \da
\end{align*}
where we compute $\bar g^\intercal \cdot DA(h)=0$ to get the last identity. Thus, we show that $DA(h)=0$ on $\Sigma^+$. 

To summarize our argument, we have shown that for any $a\in J$ and for any $(h, v)\in \Ker L_{a}$, we must have $DA|_{g_a}(h)=0$ on $\Sigma^+_a$.

Applying $DA(h)=0$ on $\Sigma^+$ to \eqref{eq:boundary}, the static vacuum deformation $(k, w)=\big(p-L_X h, z - X(v)\big)$ defined earlier satisfies $k^\intercal = 0$ and $DH(k)=0$ everywhere on $\Sigma$. In particular,  $(k, w)\in \Ker L$, and thus $DA(k)=0$ on $\Sigma^+$. We show that $DA(k)= \nabla_\nu (DA(h))$ on $\Sigma^+$: Using $DA(k) = DA(p-L_X h)$ and 

\[
p-L_X h = \lim_{t_j\to 0}\frac{1}{t_j}(h(t_j)-h)-\lim_{t_j\to 0}\frac{1}{t_j}(\psi_{t_j}^*h-h)=\lim_{t_j\to 0}\frac{1}{t_j}\big(h(t_j)-\psi_{t_j}^*h\big),
\]
we compute on $\Sigma^+$:
\begin{align*}
	0=DA(p-L_Xh) &=DA\left(\lim_{t_j\to 0}\tfrac{1}{t_j}\big(h(t_j)-\psi_{t_j}^*h\big)\right)\\
	&=\lim_{t_j\to 0}\tfrac{1}{t_j} DA|_{g_{t_j}}(h(t_j)-\psi_{t_j}^*h)\\
	&=-\lim_{t_j\to 0} \tfrac{1}{t_j}DA|_{g_{t_j}}(\psi_{t_j}^*h)\\
	&=-\lim_{t_j\to 0} \tfrac{1}{t_j} \psi_t^*\big(DA(h)\big|_{\Si_{t_j}}\big)\\
	&  =-L_X\big(DA(h)\big)\\
	&=-\zeta \nabla_{\nu}\big(DA(h)\big).
\end{align*}
where in the second equality we use $h(t_j)-\psi_{t_j}^* h=0$ when $t_j=0$, in the third equality we use $DA|_{g_{t_j}}\big(h(t_j)\big)=0$ on $\Si^+_{t_j}$ because $\big(h(t_j),v(t_j)\big)\in \Ker L_{t_j}$ and  $\Sigma^+_{t_j}\to \Sigma^+$ as $t_j\to 0$, and in the last equality we use $DA(h)=0$ on $\Si^+$.

%In the third line above, we use the fact that $\big(h(t_j),v(t_j)\big)\in \Ker L_{t_j}$ so $DA|_{g_{t_j}}\big(h(t_j)\big)=0$ on $\Si_{t_j}^+$. Since the deformation vector field $X$ is smooth, $\Si_{t_j}^+\to\Si^+$ as $t_j\to 0$. Thus for any $p\in\Si^+$ there is $\epsilon>0$ such that, for all $t_j<\epsilon$, $DA|_{g_{t_j}}\big(h(t_j)\big)=0$ at $p$. In the last line above, we apply the definition of Lie derivative and use the fact that $DA(h)=0$ on $\Si^+$.

To conclude the proof, we computed as in \eqref{eq:normal} to get
\[
	 D (\nabla_\nu A)(h) = \nabla_\nu (DA(h)) = 0 \quad \mbox{ on } \Sigma^+.
\]
 
\end{proof}

\begin{proof}[Proof of Theorem~\ref{th:one-sided}]
Let $\{ \Sigma_t \}$, $t\in [-\delta, \delta]$,  be given as in the theorem. Let $J$ be the open dense subset of $J$ from Proposition~\ref{pr:limit}. We will show that, for any $a\in J$,  $\Sigma_a$ is static regular in $\mathbb R^n\setminus \Omega_a$. Let $(h, v)\in \Ker L_a$. We apply Theorem~\ref{th:normal}  to see that $DA|_{g_a}(h)=0$ and $D(\nabla_\nu A)|_{g_a}(h)=0$ on $\Sigma^+$ and hence on $\hat \Sigma$. Then we can apply Theorem~\ref{th:trivial} to conclude that $DA(h)=0$ on the entire $\Sigma$. It completes the proof. 
\end{proof}

\subsection*{Acknowledgement}

The second author was partially supported by the NSF CAREER Award DMS-1452477 and NSF DMS-2005588.

%\bibliographystyle{amsplain}
%\bibliography{2020}

\providecommand{\bysame}{\leavevmode\hbox to3em{\hrulefill}\thinspace}
\providecommand{\MR}{\relax\ifhmode\unskip\space\fi MR }
\providecommand{\MRhref}[2]{%
  \href{http://www.ams.org/mathscinet-getitem?mr=#1}{#2}
}
\providecommand{\href}[2]{#2}

\end{document}